\documentclass[12pt]{amsart}
\usepackage{amsmath, amssymb, amsthm}
\usepackage[margin=1in]{geometry}

\usepackage{graphicx, color}
\usepackage{subcaption}

\author{Nikola Kamburov and Boyan Sirakov}

\address{Nikola Kamburov, Facultad de Matem\'aticas, Pontificia Universidad Cat\'olica de Chile, Avenida Vicu\~na Mackenna 4860, Santiago 7820436, Chile}
\email{nikamburov@mat.uc.cl}
\address{Boyan Sirakov, Departamento de Matematica, PUC-Rio, Rua Marques de Sao Vicente 225, G\'avea, Rio de Janeiro -- CEP
22451-900, Brazil}
\email{bsirakov@puc-rio.br}

\thanks{NK was partially supported by Proyecto Fondecyt Regular No.\ 1201087. BS was partially supported by CNPq grant 310989/2018-3 and FAPERJ grant E-26/203.015/2017.}

\title[Uniform a priori estimates for the Lane--Emden system in the plane]{Uniform a priori estimates for positive solutions of the Lane--Emden system in the plane}

\newtheorem{theorem}{Theorem}[section]
\newtheorem{prop}[theorem]{Proposition}
\newtheorem{lemma}[theorem]{Lemma}

\newtheorem{remark}[theorem]{Remark}

\numberwithin{equation}{section}

\def\R{\mathbb{R}}

\def\D{\Delta}
\def\O{\Omega}
\def\Op{{\Omega^\prime}}

\def\d{\delta}

\def\l{\lambda}

\def\de{\partial}

\def\k{\kappa}

\begin{document}

\begin{abstract} We prove that positive solutions of the superlinear Lane-Emden system in a two-dimensional smooth bounded domain are bounded  independently of the exponents in the system, provided the exponents are comparable. As a consequence, the energy of the solutions is uniformly bounded. In addition, the boundedness may fail if the exponents are not comparable.
\end{abstract}

\maketitle
\bibliographystyle{plain}

\section{Introduction}
We study positive classical solutions $(u,v)\in  [C^{2}(\Omega)\cap C(\overline{\O})]^2$ of the Dirichlet problem for the celebrated Lane-Emden system
\begin{equation}\label{LEsys}
\begin{cases}
-\D u  = v^p, \quad u>0 & \text{in} \quad \O, \\
-\D v = u^q, \quad v>0 & \text{in} \quad \O, \\
u  = v = 0   & \text{on} \quad \de \O,
\end{cases}
\end{equation}
in a two-dimensional bounded domain $\O\subset \R^n$, $n=2$, with a $C^2$-smooth boundary. We deal with superlinear systems, in the sense that $p\ge1$, $q\ge1$, $pq>1$.

The Lane-Emden system is widely seen as the simplest example of coupled nonlinear elliptic PDEs. As such, it has been the object of a huge number of theoretical studies -- we will not attempt an exhaustive bibliography, referring instead to the  surveys \cite{deFig_oldsurvey}, \cite{deFig_newsurvey}, and the book \cite[Chapter 31]{quittner2019superlinear}, where large lists of references can be found, as well as a lot of information about the available results on existence, non-existence, and qualitative properties of solutions of this system and its generalizations.

As far as the solvability of the Lane-Emden system is concerned, it has been known since the founding papers \cite{ClemDeFigMit}, \cite{HulVor}, \cite{Mit} (see also \cite{DeFigORuf} for stronger results in two dimensions) that solutions exist in a $n$-dimensional smooth bounded domain, provided $(p,q)$ is below the so-called critical hyperbola, that is,
\begin{equation}\label{crithyp}
\frac{1}{p+1} +\frac{1}{q+1}>\frac{n-2}{n},
\end{equation}
while no solutions exist for domains with sufficiently simple geometry, such as star-shaped domains, when \eqref{crithyp} is violated. It is then clearly of interest to understand how solutions $(u_{p,q},v_{p,q})$ behave when the point $(p,q)$ approaches the critical hyperbola from below. For $n\ge3$ this question was studied in detail in \cite{Guerra} and \cite{ChoiKim}, where it was shown that, among other things, solutions must blow up close to the hyperbola.

The two-dimensional case is special, since the hyperbola goes to infinity as $n\to2$ in \eqref{crithyp}, and for $n=2$ solutions exist for $(p,q)$ in the whole quarter-space $p\ge1$, $q\ge1$, $pq>1$. The corresponding asymptotic regimes one needs to study then are the cases when at least one of $p,q$ tends to infinity.

Here we address the fundamental question of {\it uniform boundedness} of solutions. For any {\it fixed} $(p,q)$ under the critical hyperbola the solutions of \eqref{LEsys}  are bounded by a constant which depends on $(p,q)$. This is very well known and implies existence via fixed-point methods and degree theory -- see \cite[Theorem 31.2]{quittner2019superlinear}, as well as \cite{MitPoh}, \cite{deFig_newsurvey},  \cite{QuitSoup}, \cite{DeFigSir}, and the references therein to various developments.

What we study here is whether solutions are bounded {\it independently of} $p,q$. This cannot be the case for $n\ge3$ when $(p,q)$ is close to the critical hyperbola, or else a limiting procedure would yield the existence of a solution for a point $(p,q)$ on the hyperbola. On the other hand, for $n=2$  there is no such argument, and it turns out that the question is quite challenging.

While obviously important in itself, our research was triggered by a recent preprint  of Z. Chen, H. Li and W. Zou \cite{ChenLiZou}, in which they obtain a rather complete description of the  behavior of solutions of \eqref{LEsys} in a smooth bounded domain of the plane, in the asymptotic regime
\begin{equation}\label{chinregime}
 p\to \infty, \qquad |p-q|\le C_0,
 \end{equation}
for a given absolute constant $C_0$. To prove their results, they assume
 an additional integral bound: for a constant $C$ independent of $u,v,p,q$,
\begin{equation}\label{intgCond}
p \int_{\O} \nabla u\cdot \nabla v \, dx \leq C \qquad \text{for all large } p.
\end{equation}
It is shown in \cite{ChenLiZou} that this condition is valid for the least energy solutions of the system, and thus the asymptotic behavior under \eqref{chinregime} of these solutions is deduced.

As a corollary to our main result  (Theorem \ref{main} below) on the uniform boundedness of solutions, we will show that  \eqref{intgCond} can be completely removed in \cite{ChenLiZou}, and hence the asymptotic analysis there is valid for arbitrary positive solutions,  in  star-shaped domains. Actually  \eqref{intgCond} is true for arbitrary solutions of \eqref{LEsys}, provided that $p\sim q$ at infinity and the domain $\Omega$ is star-shaped -- see Theorem \ref{coro_IntgBnd} below.

It is worth noting that the corresponding  asymptotic analysis as $p\to\infty$ for the scalar Lane-Emden equation
\begin{equation}\label{LE}
\left\{\begin{array}{lcl}
-\D u  = u^p, \quad u>0 & \text{in} & \O, \\
u  = 0  &\text{on}& \de \O,
\end{array}
\right.
\end{equation}
to which \eqref{LEsys} reduces for $p=q$, has had a long history but was completed only recently. We refer to \cite{rey1990role}, \cite{han1991asymptotic}, \cite{ren1994two}, \cite{adigrossi}, \cite{Italy_Asymp}, \cite{Italy_Quant}, \cite{Thizy}, \cite{Italy_Survey} for  a very complete picture of the blow-up profiles of the solutions of \eqref{LE}, when $p\to\infty$. The studies in two dimensions depended on the integral condition \eqref{intgCond} with $u=v$, $p=q$; we contributed to that study in \cite{KS2018}, where we proved that positive solutions of the Lane-Emden equation \eqref{LE} are uniformly bounded as $p\to \infty$, so the integral condition is always satisfied for such solutions, in star-shaped domains. In the recent paper \cite{Italy_Uniq} the results from \cite{lin1994uniqueness}, \cite{Italy_Asymp} and \cite{KS2018} were used to prove the uniqueness of positive solutions of the scalar Lane-Emden equation in a convex domain, for sufficiently large values of $p$.

We now give our main result on the Lane-Emden system.

\begin{theorem}\label{main} Let $\Omega\subset\R^2$ be a bounded domain with $C^2$-boundary. Suppose that
\begin{eqnarray}
& p\ge1, \quad q\ge1,\quad pq-1 \geq \k, \label{eq:expocond1} \\
& p \leq K q \quad(\mbox{resp.\ } q\le Kp),\label{eq:expocond2}
\end{eqnarray}
for some  constants $\k >0$ and $K\geq 1$. Then there exists a constant $C>0$, depending only on $\k$, $K$ and $\Omega$, such that  the component $u=u_{p,q}$ (resp.\ $v=v_{p,q}$) of any classical solution of \eqref{LEsys} satisfies
\[
\|u\|_{L^{\infty}(\Omega)} \leq C \quad(\mbox{resp.\ } \|v\|_{L^{\infty}(\Omega)}\leq C ).
\]
In particular, if
\begin{equation}\label{eq:pqcomparable}
\frac{1}{K}q \leq p \leq K q,
\end{equation}
then both solution components $u=u_{p,q}$ and $v=v_{p,q}$ are uniformly bounded by $C$.
\end{theorem}

Hypothesis \eqref{eq:expocond1} is a superlinearity assumption -- if $p=q=1$ we have an eigenvalue problem, whose solutions are not bounded, when they exist. The really important restriction in Theorem \ref{main} is \eqref{eq:expocond2}, resp.\ \eqref{eq:pqcomparable}. It is largely sufficient for the main application we have in mind -- the bound \eqref{intgCond} under \eqref{chinregime} and the resulting asymptotic analisys in \cite{ChenLiZou}.  Note that \eqref{chinregime} implies \eqref{eq:pqcomparable} for any $K>1$.

At first glance one may think \eqref{eq:pqcomparable} is technical. Our next theorem shows that this is not so, for in the ``extreme" case $p=1$ (when the Lane-Emden system becomes the Navier problem for the biharmonic Lane-Emden equation $(-\Delta)^2u=u^q$) in a disk the $v$-component grows logarithmically as $q\to \infty$.  See also the graphs in Figures \ref{fig:image1} and \ref{fig:image2} at the end of Section~\ref{secbiham}.

\begin{theorem}\label{biheq} Let $\Omega\subset\R^2$ be the unit disk, $p=1$, and assume $q\geq 1+\k$ for some $\k>0$. Let $(u,v)=(u_q,v_q)$ be a solution of \eqref{LEsys}. Then there is a constant  $C\geq 1$ depending only on $\k$ such that
\begin{equation}\label{ubd}
\|u_q\|_{L^{\infty}(\Omega)} \leq C 
\end{equation}
while
\begin{equation}\label{vlog}
C^{-1} \log (q) \leq  \|v_q\|_{L^{\infty}(\Omega)}\leq C\,\log (q).
\end{equation}
\end{theorem}

Both the result in Theorem \ref{biheq} and the difficulty of its proof are surprising. We consider this theorem a compelling piece of evidence of the richer nature and  greater complexity of the Lane-Emden system when compared to the scalar equation.

 We conjecture that if we assume only \eqref{eq:expocond1} then the maxima of the solution components are always bounded by $C(\k,\O)\log(pq)$, for any smooth bounded $\Omega$ (the proof of Theorem~\ref{biheq} implies this for $p=1$, see Theorem \ref{thm:v_blowup} below). Finding the optimal relation between $p$ and $q$ which guarantees that the solutions of the Lane-Emden system are bounded independently of $(p,q)$ as in Theorem~\ref{main} is an interesting and delicate  open problem.
\medskip

The proof of Theorem \ref{main} employs the method we introduced in \cite{KS2018} for the scalar case.  The main ingredients of that method are the Green's representation formula, the inhomogeneous Harnack inequality, the natural $L^p$-bound for $u$ which can be deduced from the equation, and a rescaling argument. The main idea of the method is to bound from below by a positive constant the term $u^p$ in the singular integral in the Green identity, written at a maximum point of $u$, by showing that $u$ remains close to its maximum on a sufficiently large (though small) ball around the point where the maximum is attained -- and we ensure this with the help of the Harnack inequality. Then the logarithmic singularity provides the desired bound.

It is curious that when applied to the scalar equation in \cite{KS2018}, the method feels somewhat ``overspecified", most saliently in that the rescaling provides an equation in a large ball, while the estimate of $u$ only happens in a small ball around the maximum; also, the $L^p$-estimate is used only to control the regular part of the Green function in the representation formula.

The system case is much more delicate because of the coupling of $u$ and $v$, the effect that the two different exponents $p$ and $q$ have on the natural scales associated to the two solution components,  and the fact that the latter may achieve their maxima at different points. For the proof of Theorem \ref{main} we need the full strength of the method from \cite{KS2018}: the Harnack inequality is applied in balls of a priori unknown size, only just adjusted to the size of the rescaled domain. The radii of the balls are precisely determined by the maximum principle and the resulting comparison between the maxima of the solution components. The Harnack inequality is applied not only in the Green formula but also in the $L^p$ (resp.\ $L^q$) bounds for the solutions, and only the combined strength of the resulting estimates permits us to uncouple the relations between the maxima of the components.\medskip

The proof of Theorem \ref{biheq} is even more complex than that of Theorem \ref{main}, and uses heavier tools. Specifically, we rely on the recent global Harnack inequality from \cite{sirakov2022global} (see below), on the Brezis-Merle exponential bound for the Dirichlet problem from \cite{brezis1991uniform}, and the inhomogeneous Harnack inequality with an unbounded right-hand side. We prove that the $L^1$ norms of $u$, $u^{q+1}$, $v$, and $v^{p+1}$ stay uniformly bounded away from zero (in contrast to what happens for comparable $p$ and $q$),  then that the solutions decrease quadratically  away from their maxima, and  combine these facts with delicate evaluations of the measures of the superlevel sets of the components. We note that the proof of Theorem \ref{biheq} is entirely PDE-based; we use the fact that $u,v$ attain their maximum at the same point and decrease away from it, but we do not rely on the equivalent ODE formulation.\medskip

As we already noted, it is not our aim here to give a thorough account of the reasons for which studying systems is much more complicated and challenging than scalar equations -- we refer to the surveys and books cited above. One distinction we would like to mention concerns the variational formulation of the system, since it is important in the proof of our last result. Specifically, when searching for (finite energy) solutions of \eqref{LEsys} as critical points of the functional
$$
J(u,v) = \frac{1}{2}\int_\Omega \nabla  u\cdot \nabla v\, dx  - \frac{1}{p+1}\int_\Omega (u^+)^{p+1} \, dx- \frac{1}{q+1}\int_\Omega (v^+)^{q+1}\, dx,
$$
we see that when $u=v$ the Dirichlet energy $\int_\Omega |\nabla u|^2$ is positive and coercive on the energy space $H^1_0(\Omega)$, while in general the term $\int_\Omega \nabla u \cdot \nabla v$ is strongly indefinite on $[H^1_0(\Omega)]^2$. This renders it impossible to derive the estimate \eqref{intgCond} from Theorem \ref{main} in the same way as the corresponding scalar estimate was derived in \cite{KS2018}, since there we used the Cauchy-Schwarz inequality to bound the Dirichlet energy on the boundary from below. Nevertheless, we prove  \eqref{intgCond} as follows.





\begin{theorem}\label{coro_IntgBnd} Let $\O\subset \R^2$ be a star-shaped bounded domain with $C^2$-boundary and assume that the exponents $p,q \geq 1$ satisfy \eqref{eq:expocond1} and \eqref{eq:pqcomparable} for some $\k>0$ and $K\geq 1$. There exists a constant $C$, depending only on $\k$, $K$ and $\Omega$, such that  any classical solution $(u,v)=(u_{p,q}, v_{p,q})$ of \eqref{LEsys} satisfies
\[
p \int_{\O} \nabla u \cdot \nabla v \, dx = p \int_{\O} v^{p+1} =  p \int_{\O} u^{q+1} \, dx \leq C.
\]
\end{theorem}
The proof of Theorem \ref{coro_IntgBnd} is based on a Pohozaev identity, which is the reason behind the hypothesis on the domain being star-shaped. The new key ingredient of the proof is a recent \emph{global Harnack inequality} due to the second author (\cite[Theorem 1.3]{sirakov2022global}), which allows us to estimate from below the absolute normal derivatives of the solution components in terms of their $L^1$-norms.

We observe that Theorem \ref{coro_IntgBnd} is also an improvement to the result from \cite{KS2018} since we can consider a non-strictly star-shaped domain. That our proof of Theorem \ref{coro_IntgBnd} actually implies this fact was pointed out to us by Z.\ Chen.

Theorem \ref{coro_IntgBnd} has already been exploited in \cite{ChenLiZou_uniq}, to obtain sharp estimates on solutions of the Lane-Emden system \eqref{LEsys} in the asymptotic regime \eqref{chinregime}, and deduce the uniqueness of solutions in a convex domain under \eqref{chinregime} for sufficiently large $p$, an extension to systems of the result in \cite{Italy_Uniq}.

In the next section we prove Theorem \ref{main} and Theorem \ref{coro_IntgBnd}. The last section is devoted to the proof of Theorem \ref{biheq}.

\section{Proofs}

\subsection{Preliminaries}
In what follows the letters $C,c$ (possibly with indices and primes) will denote positive constants which depend only on $\k$, $K$ and $\Omega$, and which may change from line to line. For the sake of notational simplicity, we shall drop the subscripts $p,q$ from the solution components $u_{p,q}$, $v_{p,q}$ and their correponding features. Throughout the exposition we shall denote with
\begin{equation*}
M:= \max_{\overline{\O}} u, \qquad N:= \max_{\overline{\O}}v.
\end{equation*}
the maxima of $u$ and $v$, respectively. The ball of radius $r$, centered at $x\in \R^n$, is denoted by $B_r(x)$ and $B_r:=B_r(0)$.

We start with the following classical integral estimate, whose proof can be extracted, for instance, from \cite[Theorem 31.2]{quittner2019superlinear} and \cite{de1982priori}.

\begin{prop}\label{prop:int} Let $(u,v)$ be a classical solution of \eqref{LEsys} in a bounded $C^2$-domain $\O\subset \R^2$, and \eqref{eq:expocond1} holds.
There exist positive constants $\d$ and $c$ depending only on $\Omega$, and $C$ depending only on $\k$ and $\Omega$, such that
\begin{itemize}
\item The maxima of $u$ and $v$ in $\overline{\Omega}$ are attained in $\{x\in \Omega:\text{dist}(x,\de\Omega)> \d\}$;
\item We have the bounds
\begin{align}
\int_{\O}v^p \, dx  & \leq c \|u \phi\|_{L^1(\O)} \leq C \label{prop:int:eq1}, \\
\int_{\O}u^q \, dx  & \leq c  \|v \phi\|_{L^1(\O)}  \leq C, \label{prop:int:eq2}
\end{align}
where $\phi$ is the eigenfunction of the Dirichlet Laplacian, associated with the lowest eigenvalue $\lambda=\lambda(\O)$:
\begin{equation*}
-\D \phi = \lambda \phi, \quad \phi>0\mbox{ in } \Omega, \qquad \phi = 0 \text{ on } \de\Omega, \quad \text{normalized so that} \quad \|\phi\|_{L^1(\Omega)}=1.
\end{equation*}
\end{itemize}
\end{prop}
\begin{proof}
For the reader's convenience we shall provide a sketch of the proof, displaying the dependence of the constants on the exponents $p,q$.

Let $\O_\d:=\{x\in \O:  \text{dist}(x,\de\Omega)> \d\}$. By using the moving planes technique for semilinear elliptic systems, in combination with the Kelvin transform (see Step 2 of the proof of Theorem 31.2 and  Remark 31.5(ii) in \cite{quittner2019superlinear}, as well as the considerations after Theorem 1.2 in \cite{de1982priori} for the case $n=2$), one can show the existence of positive constants $\d$ and $\gamma$ depending only on~$\O$, such that
$$
 \text{ for each } x\in \O\setminus \O_{2\d} \quad  \text{ there exists a measurable set } I_x \text{ such that}
$$
\begin{equation}\label{prop:int:MovPlane}
\begin{array}{rl}
(i) & I_x \subseteq \O_\d \\
(ii) & |I_x|\geq \gamma \\
(iii) & u(x)\leq u(\xi) \quad \text{and}\quad v(x)\leq v(\xi)\quad \text{for all } \xi\in I_x.
\end{array}
\end{equation}
For instance, when $\Omega$ is convex, $I_x$ can be taken to be a part of a cone with a vertex at~$x$. Non-convex domains can be treated by Kelvin inversion of neighborhoods of boundary points in which the boundary is not convex.

In particular, we know that the maxima of $u$ and $v$ are achieved some unit distance away from the boundary of $\O$ -- in $\O_\d$.

For the second part of the proposition, we argue as in Step 1 of \cite[Theorem 31.2]{quittner2019superlinear}, additionally tracking the dependence on $p$, $q$ and $\O$. Multiplying each equation in \eqref{LEsys} by the Dirichlet eigenfunction $\phi$ and integrating by parts,  we get
\begin{equation}\label{prop:int:parts}
\lambda \int_{\O} u\phi \, dx= \int_{\O} v^p \phi \, dx \quad \text{and} \quad \lambda \int_{\O} v \phi \, dx= \int_{\O} u^q \phi \, dx.
\end{equation}
Furthermore, by \eqref{prop:int:MovPlane} we see that for all $y\in \O\setminus \O_{2\d}$
\begin{equation}\label{prop:int:key}
\int_{\O} u^q \phi ~dx \geq \int_{I_y} u^q \phi ~dx \geq \gamma u^q(y) \inf_{I_y}\phi \geq \gamma u^q(y) \inf_{\O_\d} \phi,
\end{equation}
as $I_y\subseteq \O_\d$.
The  Harnack inequality for $\Delta \phi+\lambda\phi=0$, $\phi>0$, gives us $ \inf_{\O_\d} \phi = c_0>0$ with $c_0$ depending only on $\Omega$, so that integrating the estimate in \eqref{prop:int:key} over $y\in \O\setminus \O_{2\d}$, we obtain
\[
 \gamma c_0 \int_{\O\setminus \O_{2\d}} u^q(y)\, dy \leq |\O| \int_{\O} u^q \phi \, dx.
\]
Thus,
\begin{align*}
\int_\O u^q \,dx & \leq \int_{\O_{\d}} u^q \,dx +  \int_{\O\setminus \O_{2\d}} u^q\, dx  \leq \int_{\O_\d} \frac{u^q \phi}{\inf_{\O_\d} \phi} \, dx + \frac{|\O|}{\gamma c_0}  \int_{\O} u^q \phi \, dx  \\
& \leq \left(\frac{1}{c_0} + \frac{|\O|}{\gamma c_0} \right) \int_\O u^q \phi \, dx = \frac{|\O|/\gamma +1}{c_0} \l \int_\O v \phi = c \|v\phi\|_{L^1(\O)}
\end{align*}
by \eqref{prop:int:parts}, where $c$ depends only on $\O$.
Similarly,
\[
\int_\O v^p \,dx  \leq c  \|u \phi \|_{L^1(\O)}.
\]
Now, in order to bound the $L^1$-norms of $u\phi$ and $v\phi$, one applies the Jensen inequality to the right-hand sides of \eqref{prop:int:parts}:
\begin{equation}\label{prop:int:Jensen}
\int_{\O} v^p \phi \,dx \geq \|v\phi\|_{L^1(\O)}^p \quad \text{and}\quad \int_{\O} u^q \phi \,dx \geq \|u\phi\|_{L^1(\O)}^q.
\end{equation}
Combining \eqref{prop:int:parts} and \eqref{prop:int:Jensen} yields the desired estimates
\begin{equation*}
\begin{array}{c}
\| u\phi \|_{L^1(\O)} \leq \l^\frac{p+1}{pq-1}= \l^\frac{\frac{1}{pq} + \frac{1}{q}}{1-\frac{1}{pq}} \leq \max\left(1, \l^{2/\bar{\k}} \right) = C(\O,\k), \\
\| v \phi\|_{L^1(\O)} \leq \l^\frac{q+1}{pq-1}= \l^\frac{\frac{1}{pq} + \frac{1}{p}}{1-\frac{1}{pq}}\leq \max\left(1, \l^{2/\bar{\k}} \right) = C(\O,\k).
\end{array}
\end{equation*}
where $\bar{\k}=\k/(\k+1)$ and we used \eqref{eq:expocond1}.
\end{proof}

The next lemma is the standard $L^\infty$-estimate for the Poisson equation, applied to \eqref{LEsys}.
\begin{lemma}\label{lem:relativesizes} Let $(u,v)$ be a classical solution of \eqref{LEsys} in a bounded domain $\O\subset \R^2$, $M=\max_{\overline{\O}} u$ and $N=\max_{\overline{\O}} v$. Then there exists a positive constant $c_0$, depending only on $\text{diam}(\O)$, such that
\[
\frac{N}{M^q} \leq c_0 \quad \text{and} \quad \frac{M}{N^p} \leq c_0.
\]
\end{lemma}
\begin{proof}
Let $\tilde{u}:= u/M$ and $\tilde{v}:=v/N$. Then $\|\tilde{u}\|_{L^{\infty}(\O)} = 1=\|\tilde{v}\|_{L^{\infty}(\O)}$ and
\[
-\D \tilde{u} = (N^p/M) \tilde{v}^p \quad \text{and}\quad -\D \tilde{v} = (M^q/N) \tilde{u}^q.
\]
Without loss of generality, assume that $\tilde{u}$ achieves its maximum at $0\in \O$, $\tilde{u}(0)=1$, and consider the function
\[
h(x):=(\text{diam}(\O)^2 - |x|^2)\frac{N^p}{4 M} \quad \text{for } x\in \O.
\]
We see that $h\geq 0 = \tilde{u}$ on $\de \O$ and that
$
-\D h =  \frac{N^p}{M} \geq -\D \tilde{u}.
$
in $\O$, so $h\geq \tilde{u}$ in $\O$ by the comparison principle. In particular,
\[
1=\tilde{u}(0)\leq h(0) = \frac{\text{diam}(\O)^2}{4} \frac{N^p}{M} \quad \Longrightarrow \quad \frac{M}{N^p}\leq \text{diam}(\O)^2/4 :=c_0.
\]
Exchanging the roles of $u$ and $v$ yields the other estimate $N/M^q \leq c_0$.
\end{proof}

\subsection{Proof of Theorem \ref{main}} We divide the proof into several steps.\medskip

\noindent\emph{Step 1.} In the first step we employ the inhomogeneous Harnack inequality to estimate  how fast the values of $u$ and $v$ can decrease away from their maxima.
\begin{lemma}\label{lem:inhHarnack} Let $(u,v)$ be a classical solution of \eqref{LEsys} in a domain $\O\subseteq \R^2$. Assume that $M=\max_{\overline{\O}} u$ is attained at $x_u\in \O$ and $N=\max_{\overline{\O}} v$ is attained at $x_v\in \O$. Fix $r>0$. Then
\begin{equation}\label{lem:inhHarnack:1}
M - u(x) \leq C_0 N^{p} r^2 \quad\text{for all } x\in B_r(x_u), \quad \text{provided }\; B_{2r}(x_u)\subseteq \O,
\end{equation}
for an absolute constant $C_0>0$. Similarly,
\begin{equation}\label{lem:inhHarnack:2}
N - v(x) \leq C_0 M^{q} r^2 \quad\text{for all } x\in B_r(x_v), \quad \text{provided }\; B_{2r}(x_v)\subseteq \O.
\end{equation}
\end{lemma}
\begin{proof}
Let $\rho:=r N^{(p-1)/2}$. Define the function
\begin{align*}
\hat{u}(y) &:=\frac{M- u(x_u+ y N^{-\frac{p-1}{2}})}{N} \quad \text{for} \quad  y\in B_{2\rho},
\end{align*}
and note that $\hat{u}\geq 0$ in  $B_{2\rho}$ while $\hat{u}(0)=0$. Moreover,
\[
\D \hat{u}(y) = - N^{-p} \Delta u(x_u + y N^{-(p-1)/2}) = \left(\frac{v}{N}\right)^p(x_u + y N^{-(p-1)/2}) \quad
\text{for } y\in B_{2\rho},
\]
 so that $|\D\hat{u}|\leq 1$ in $B_{2\rho}$. Applying the inhomogeneous Harnack inequality (see for instance Theorem 4.17 in \cite{HanLin}) to $\hat{u}$ in $B_{2\rho}$, we deduce that
\[
\sup_{B_{\rho}} \hat{u} \leq C_0(\inf_{B_\rho} \hat{u} + \rho^2 \|\D\hat{u}\|_{L^{\infty} (B_{2\rho})}) = C_0\,\rho^2,
\]
from which we conclude that
\[
\sup_{B_{r}(x_u)}(M-u(x))=  N \sup_{B_{\rho}} \hat{u} \leq C_0N\rho^2 = C_0 N^p r^2.
\]
The estimate \eqref{lem:inhHarnack:2} for $v$ is obtained analogously.
\end{proof}

\noindent\emph{Step 2.} In the second step we will use Lemma \ref{lem:inhHarnack} to estimate $u^q$ and $v^p$ from below around the points of maximum $x_u$ of $u$ and $x_v$ of $v$, respectively ($M=u(x_u)$ and $N=v(x_v)$). \smallskip

{\it We claim that if
$$
R_1=\left(\frac{M}{qN^p}\right)^{1/2}, \qquad R_2=\left(\frac{N}{pM^q}\right)^{1/2},
$$
then  for some constant $\hat{C}=\hat{C}(\O,\k)$ we have $\hat{C}R_1<1$, $\hat{C}R_2<1$, and
\begin{equation}\label{main:vlow}u^q\geq  e^{-1/3} M^q\;\mbox{  in a ball  }\;B_{\hat C R_1}(x_u)\subset\Omega,\quad\; v^p\geq e^{-1/3} N^p\;\mbox{  in a ball }\;B_{\hat CR_2}(x_v)\subset\Omega.
\end{equation}}

\noindent{\it Proof.}
By the first part of Proposition \ref{prop:int}, we know that $d(x_u, \de \O)$ and $d(x_v, \de \O)$ are both greater than  $\d=\d(\O)>0$.  Furthermore, the result of Lemma \ref{lem:relativesizes} implies that
if
\[
r_1 := \frac{\min(\d,1)}{2\sqrt{c_0}}\left(\frac{M}{N^p}\right)^{1/2} \quad \text{and} \quad r_2 := \frac{\min(\d,1)}{2\sqrt{c_0}}\left(\frac{N}{M^q}\right)^{1/2}
\]
then  $r_1, r_2 \in (0,\min(\d,1)/2)$ so that
\[
 B_{2r}(x_u)\subset \O \quad \text{for all } r\in(0,r_1), \quad  B_{2r}(x_v)\subset \O \quad \text{for all } r\in(0,r_2).
\]
 Set $\bar{c}:= \max(1, \sqrt{C_0\d^2/c_0)})$, where $C_0$ is the constant from Lemma \ref{lem:inhHarnack}.  We can therefore apply the $u$-estimate in Lemma \ref{lem:inhHarnack} for $r=r_1/(\bar{c}\sqrt{q}) \leq r_1 <1$, obtaining by the definition of $r_1$
\begin{equation}\label{main:uest}
M- u \leq C_0 N^p \frac{r_1^2}{q}  \leq \frac{C_0\d^2}{4c_0 \bar{c}^2} \frac{M}{q} \leq \frac{M}{4q} \quad \text{in} \quad B_{r_1/(\bar{c}\sqrt{q})}(x_u),
\end{equation}
which implies that near the point of maximum of $u$,  by using that $\log(1-x)\geq (-4/3)x$ for $x\in (0,1/4)$,
\begin{equation}\label{main:bduq}
u^q(y) \geq M^q(1-1/(4q))^q  \geq e^{-1/3} M^q \quad \text{for } y\in B_{r_1/(\bar{c}\sqrt{q})}(x_u)\subset \O.
\end{equation}
Similarly, the $v$-estimate of Lemma \ref{lem:inhHarnack} for $r=r_2/(\bar{c}\sqrt{p})\leq r_2 <1$ yields
\begin{equation}\label{main:vest}
N- v \leq C_0 M^q \frac{r_2^2}{p} \leq \frac{N}{4p}\quad \text{in} \quad B_{r_2/(\bar{c}\sqrt{p})}(x_v),
\end{equation}
which implies that near the point of maximum of $v$,
\begin{equation}\label{main:bdvp}
v^p(y) \geq N^p(1-1/(4p))^p \geq  e^{-1/3} N^p \quad \text{for } y\in B_{r_2/(\bar{c}\sqrt{p})}(x_v)\subset \O.
\end{equation}

\medskip

\noindent \emph{Step 3.} We will now obtain some initial crude estimates relating $M$ and $N$, applying the pointwise bounds from the previous step to the integral estimates \eqref{prop:int:eq1}--\eqref{prop:int:eq2} of Proposition~\ref{prop:int}.

{\it We claim that
\begin{equation}\label{main:MNM}
N\leq C M^{\frac{q}{p+1}}, \qquad M \leq C N^{\frac{p}{q+1}}.
\end{equation}
}

\noindent{\it Proof.} Combining \eqref{prop:int:eq2} with \eqref{main:vlow} we get
\begin{equation}\label{main:NM}
C\geq \int_\O u^q \, dx \geq \int_{B_{\hat CR_1}(x_u)} u^q \, dx \geq  e^{-1/3} M^q |B_{\hat CR_1}(x_u)| =cM^qR_1^2= c' \frac{M^{q+1}}{q N^p},
\end{equation}
which gives us the bound
\begin{equation*}
M \leq (Cq/c')^{1/(q+1)} N^{\frac{p}{q+1}}\leq C N^{\frac{p}{q+1}}.
\end{equation*}
The analogous estimate for $N$ reads
\begin{equation*}
N \leq c M^{\frac{q}{p+1}}.
\end{equation*}

\noindent \emph{Step 4.} {\it Without loss of generality, we may assume that $M\geq A$ and $N\geq A$, for  some large constant $A=A(\O,K,\k)\geq 1$ (which will be chosen later in the proof)}.

We recall that the statement of the theorem is $M\le C=C(\O,K,\k)$. Hence we can assume $M\ge A$.
 Furthermore, if $N< A$, then \eqref{main:MNM} would yield the desired bound for $M$: $$M\leq C A^{\frac{p}{q+1}}\leq C  A^{\frac{K q}{q+1}}\leq C A^K \quad \text{since} \quad p\leq Kq$$
(this is where we use the hypothesis $p\leq Kq$).
\bigskip

\noindent \emph{Step 5.} In this step we will apply the key estimates \eqref{main:vlow} to the Green's representation formula, precisely for the values of $u$ at $x_v$ and of $v$ at $x_u$.

{\it We claim that
\begin{equation}\label{main:MNMN}
\frac{(M+C')M^q}{N^{p+1}} \geq \frac{C}{p} \log \left(c \frac{p M^q}{N}\right), \qquad \frac{(N+C')N^p}{M^{q+1}} \geq \frac{C}{q} \log \left(c \frac{q N^p}{M}\right).
\end{equation}
}

\noindent{\it Proof.} Denote by $G(x,y)$ the Green's function for the Laplacian $-\D$ in $\O$, that is,
\[
G(x,y) = \frac{1}{2\pi} \log\frac{1}{|x-y|} - g(x,y),
\]
where for each fixed $x\in \O$, $g(x,\cdot)$ is harmonic in $\O$ with boundary data
\[
g(x,y) =  \frac{1}{2\pi} \log\frac{1}{|x-y|}, \quad y\in \de \O.
\]
Since $|x_v - y|\geq \d$ for $y\in \de \O$ we see that $g(x_v, \cdot) \leq C$ on $\de\O$, so that the maximum principle yields
\begin{equation}\label{main:Green_rem}
g(x_v,y) \leq C \quad \text{for all } y\in \O.
\end{equation}
Applying now Green's representation formula,  and using the positivity of $G$, we see that for each $\Op\subseteq\O$
\begin{align}
M \geq u (x_v)  &= \int_\O G(x_v,y) v^p(y) \, dy \ge \int_\Op \frac{1}{2\pi} \log\frac{1}{|x_v-y|} v^p(y)\, dy - \int_\Op g(x_v,y) v^p(y) \, dy \notag \\
& \geq \int_\Op \frac{1}{2\pi} \log\frac{1}{|x_v-y|} v^p(y)\, dy - C \int_\Op v^p dy \notag \\
& \geq \int_\Op \frac{1}{2\pi} \log\frac{1}{|x_v-y|} v^p(y)\, dy - C', \label{main:Green_u}
\end{align}
where in the last two lines we used \eqref{main:Green_rem} and the bound \eqref{prop:int:eq1} of Proposition \ref{prop:int}.

 Next we observe that since $\hat{C}R_2 < 1$, the logarithmic term in \eqref{main:Green_u} is positive if we take $\O'=B_{\hat CR_2}(x_v)$.
We can thus utilize the pointwise bound \eqref{main:vlow} for $v^p$ to estimate this integral from below: 
\begin{align}
\int_\Op \frac{1}{2\pi} \log\frac{1}{|x_v-y|} v^p(y)\, dy & = \int_{B_{\hat CR_2}(x_v)} \frac{1}{2\pi} \log\frac{1}{|x_v-y|} v^p(y) \geq c' N^p \int_{0}^{\hat CR_2} \log\left(\frac{1}{r}\right)\, r \, dr  \notag \\
& \geq c' N^p R_2^2 \log \frac{\tilde{c}}{R_2}= \frac{N^{p+1}}{M^q} \frac{C}{p} \log \left(c \frac{p M^q}{N}\right)\label{main:logterm},
\end{align}
where we used the inequality $$\int_0^\rho \log(1/r) r \, dr \geq \frac{1}{2}\rho^2 \log(1/\rho)$$ and the fact that $R_2^2 = N/(pM^q)$. Combining \eqref{main:Green_u} and \eqref{main:logterm}, we obtain
\begin{equation}\label{main:ineq_u}
\frac{(M+C')M^q}{N^{p+1}} \geq \frac{C}{p} \log \left(c \frac{p M^q}{N}\right).
\end{equation}
Analogously, exchanging the roles of $u$ and $v$, we obtain the mirror estimate:
\begin{equation}\label{main:ineq_v}
\frac{(N+C')N^p}{M^{q+1}} \geq \frac{C}{q} \log \left(c \frac{q N^p}{M}\right).
\end{equation}

\noindent \emph{Step 6. Conclusion.} Plugging the upper bound \eqref{main:MNM} for $N$ into \eqref{main:ineq_u}, we see that
\begin{equation*}
\frac{(1+C'/M)M^{q+1}}{N^{p+1}} \geq \frac{C}{p} \log \left(\tilde{c} p M^{q(1-\frac{1}{p+1})}\right) \geq \frac{C}{p} \log \left(\tilde{c} M^{q(1-\frac{1}{p+1})}\right).
\end{equation*}
Note that the assumption $M\geq A$ from Step 4 implies that $C'/M\leq 1$ if $A\geq C'$ and that
\[
\tilde{c} M^{q(1-\frac{1}{p+1})} \geq \tilde{c} M^{q/2} \geq M^{q/4} \quad\text{if}\quad A^{q/4} \tilde{c} \geq 1,
\]
i.e. provided $A\geq C_1'  := \min\{1,\tilde c\}^{-4}\geq  (\tilde{c}^{-1})^{4/q}$.
The mirror analysis involving \eqref{main:MNM} and \eqref{main:ineq_v}  yields
\begin{equation*}
\frac{(1+C'/N)N^{p+1}}{M^{q+1}} \geq \frac{C}{q} \log \left(\tilde{c}N^{p(1-\frac{1}{q+1})}\right),
\end{equation*}
and the lower bound $N\geq A$ from Step 4 implies that $C'/N\leq 1$ if $A\geq C'$ and that
\[
\tilde{c} N^{p(1-\frac{1}{q+1})} \geq \tilde{c} N^{p/2} \geq N^{p/4} \quad\text{if}\quad A^{p/4}\tilde{c} \geq 1,
\]
i.e. provided $A\geq C_2' \geq (\tilde{c}^{-1})^{4/p}$.
Hence, choosing $A$ to be
\[
A:=\max\left(2, C', C'_1, C'_2\right)
\]
we can conclude that
\begin{equation}\label{main:final}
2 M^{q+1}/N^{p+1} \geq \frac{C}{p} \log M^{q/4} \quad \text{and} \quad 2 N^{p+1}/M^{q+1}  \geq \frac{C}{q} \log N^{p/4}.
\end{equation}
Multiplying the two inequalities in \eqref{main:final}, we obtain
\[
4 \geq \frac{C^2}{pq} \log M^{q/4} N^{p/4} = c\log M \log N,
\]
and since $\log N \geq \log A \geq \log 2$, we can conclude that
\[
\log M \leq \frac{4}{c \log N} \leq \frac{4}{c\log 2}=C.
\]

This concludes the proof of Theorem \ref{main}.

\subsection{Proof of Theorem  \ref{coro_IntgBnd}}
The proof of the integral bound \eqref{intgCond} employs the \emph{global Harnack inequality} for supersolutions to second-order elliptic PDE in divergence form from~\cite{sirakov2022global}. For the reader`s convenience, we state a version of this Harnack inequality in the setting of superharmonic functions, relevant for our context.
\begin{theorem}(\cite[Theorem 1.3]{sirakov2022global})\label{thm:GHI} Assume $u\in H^1_0(\O)$ is a nonnegative weak solution to $-\D u\geq 0$ in a bounded $C^{1,1}$-domain $\O\subset \R^n$. Then for each $t <\frac{n}{n-1}$,
\[
\inf_\O \frac{u}{d} \geq C\|u\|_{L^t(\O)},
\]
where $d(x):=d(x,\de \O)$ and the constant $C>0$ depends on $\O$, $t$, and  $n$.
\end{theorem}

\begin{proof}[Proof of Theorem \ref{coro_IntgBnd}]\
Without loss of generality, we may assume that $\O$ is star-shaped with respect to the origin:
\begin{equation}\label{star}
x \cdot \nu(x) \geq 0 \quad \text{for all } x\in \de\O,
\end{equation}
where $\nu=\nu(x)$ denotes the unit outer normal to $\de \O$ at $x\in\de \O$.
We will use the following well known Pohozaev identity for the Lane-Emden system (see \cite{Mit} or \cite[Lemma 31.4(ii)]{quittner2019superlinear}): any classical solution $(u,v)$ of \eqref{LEsys} in a bounded $C^2$-domain $\O\subset \R^2$  satisfies
\begin{equation}\label{pohoz1}
 \frac{2}{p+1}\int_{\O} v^{p+1} \, dx + \frac{2}{q+1}\int_{\O} u^{q+1} \, dx = \int_{\de \O} (x \cdot \nu) u_{\nu} v_\nu \, ds.
\end{equation}
We also know via integration by parts that
\begin{equation}\label{int_part}
\int_\O v^{p+1} \, dx = \int_\O v (-\D u)  \, dx  = \int_\O \nabla u \cdot \nabla v  \, dx  = \int_\O (-\D v) u  \, dx = \int_\O u^{q+1} \, dx,
\end{equation}
so that we may rewrite the left-hand side of \eqref{pohoz1} as
\begin{equation}\label{pohoz2}
\left(\frac{2}{p+1} + \frac{2}{q+1}\right) \int_\O v^{p+1} \, dx =   \int_{\de \O} (x \cdot \nu) u_{\nu} v_\nu \, ds.
\end{equation}
Using the result of Theorem \ref{main} and that the exponents $p\sim q$ in the sense of $p/K \leq q \leq K p$, we can bound the left-hand side of \eqref{pohoz2} by
\begin{align}
\left(\frac{2}{p+1} + \frac{2}{q+1}\right) \int_\O v^{p+1} \, dx & \leq \left(\frac{2}{p} + \frac{2}{p/K} \right)\int_\O v^{p+1} \, dx \notag \\
&\leq \frac{2K+2}{p} \|v\|_{L^{\infty}(\O)}\int_\O v^p\, dx \leq \frac{C}{p}\int_{\O} v^p \, dx, \label{LHS}
\end{align}
where $C = C(\O, K, \k)$.

In order to estimate the right-hand side of \eqref{pohoz2} from below, we shall apply the global Harnack inequality of Theorem \ref{thm:GHI} to the non-negative superharmonic functions $u$ and $v$, taking the exponent $t=1<n/(n-1)=2$. Thus, we get that for every $x\in \de \O$
\begin{equation}\label{quantHopf}
-u_\nu(x) \geq \inf_{\O} u/d \geq C \|u\|_{L^1(\O)} \quad \text{and} \quad -v_\nu(x) \geq \inf_{\O} v/d \geq C \|v\|_{L^1(\O)}.
\end{equation}
Since the domain is star-shaped \eqref{star}, we have
$$
(x \cdot \nu) u_{\nu} v_\nu = (x \cdot \nu) (-u_\nu)(-v_\nu) \geq C^2 (x \cdot \nu) \|u\|_{L^1(\O)}\|v\|_{L^1(\O)},
$$
so that we can bound using the divergence theorem
\begin{align}
\int_{\de \O} (x \cdot \nu) u_{\nu} v_\nu \, ds & \geq  C^2 \left(\int_{\de \O} (x \cdot \nu) \, ds\right) \|u\|_{L^1(\O)}\|v\|_{L^1(\O)} \notag \\ & =  C^2 \left(\int_\O \text{div}(x) \, dx\right) \|u\|_{L^1(\O)}\|v\|_{L^1(\O)} =2 C^2 |\O| \|u \|_{L^1(\O)} \|v \|_{L^1(\O)}  \notag \\
& \geq c \|u\phi \|_{L^1(\O)} \|v \phi\|_{L^1(\O)} \geq c' \int_\O v^p \, dx \int_\O u^q \, dx, \label{RHS}
\end{align}
where $\phi$ is the $L^1$-normalized first Dirichlet eigenfunction of the Laplacian and we used estimates \eqref{prop:int:eq1}--\eqref{prop:int:eq2} of Proposition \ref{prop:int} to derive the last inequality. Combining \eqref{pohoz2}, \eqref{LHS} and \eqref{RHS}, we obtain
\[
\frac{C}{p} \int_\O v^p \geq c \int_\O v^p \, dx \int_\O u^q \, dx,
\]
so that
\[
p \int_\O u^q \, dx \leq C'.
\]
Now, by \eqref{int_part}, \eqref{prop:int:eq2}, and Theorem \ref{main},
\[
p\int_\O \nabla u \cdot \nabla v  \, dx  = p \int_\O u^{q+1}  \, dx \leq p \|u\|_{L^\infty(\O)} \int_\O u^q \, dx \leq C.
\]

\end{proof}


\section{The biharmonic Lane-Emden equation in the plane.}\label{secbiham}

In this section we will treat the $p=1$, $q>1$ case of \eqref{LEsys}:
\begin{equation}\label{LEbh}
\begin{cases}
-\D u  = v, \quad u>0 & \text{in} \quad \O, \\
-\D v = u^q, \quad v>0 & \text{in} \quad \O, \\
u  = v = 0   & \text{on} \quad \de \O
\end{cases}
\end{equation}
which corresponds to the Navier problem for the biharmonic Lane-Emden equation with homogeneous data. We will establish that while the $u$-maximum
$$
M_q:=\max_{\overline{\O}} u_q \to 1 \quad \text{as } q\to\infty
$$
for general bounded $C^2$-domains $\O\subset \R^2$, the $v$-maximum
$$
N_q:=\max_{\overline{\O}} v_q \sim \log q \quad \text{as } q\to\infty
$$
in the \textit{special case} of the disk $\O=B_1(0)$.  In fact, the upper bound for $v$ is valid for any smooth bounded $\Omega$.

As before, we will omit writing the $q$-dependence of the solution $(u_q,v_q)$ of \eqref{LEbh} and its features for ease of notation. We shall assume that
\begin{equation}\label{eq:qk}
q \geq 2
\end{equation}
that is $\k=1$, so all constants $C,c$ (possibly with indices and primes), including those coming from the previous proof, will depend only on $\Omega$.  When we say an inequality is valid for all large $q$, or as $q\to\infty$, we mean that it holds for $q\ge q_0$, where $q_0$ depends only on $\O$. Once the result is established for $q\ge 2$, resp.\ for $q\geq q_0$, the estimates in Theorem \ref{biheq} for the remaining range $1+\k\le q<q_0$, $\k\in(0,1)$, follow from known results on a priori bounds for fixed $p,q$, or from Theorem \ref{main} with $K=q_0$.

Because of Theorem \ref{main}, we know that in the current problem we have
\begin{equation}\label{sec:BH:M}
M \leq C.
\end{equation}
As a first consequence of the uniform boundedness of $M$, we state a modified version of the Harnack estimates from Lemma \ref{lem:inhHarnack}.
\begin{lemma}\label{lem:inhHarnack2}
Assume \eqref{eq:qk} and let $(u,v)$ be a classical solution of \eqref{LEbh} in a bounded $C^2$-domain $\O\subseteq \R^2$. Let $M=\max_{\overline{\O}} u$ be attained at $x_u\in \O$ and $N=\max_{\overline{\O}} v$ at $x_v\in \O$. Fix $r>0$. Then for an absolute constant $c_0>0$ and a constant $C=C(\O)>0$, we have
\begin{equation}\label{lem:inhHarnack2:1}
M - u(x) \leq c_0 r \|v\|_{L^2(B_{2r}(x_u))} \leq Cr \quad\text{for all } x\in B_r(x_u),
\end{equation}
provided $B_{2r}(x_u)\subseteq \O$. Similarly,
\begin{equation}\label{lem:inhHarnack2:2}
N - v(x) \leq c_0 r M^{q/2}  \|u^{q/2}\|_{L^2(B_{2r}(x_v))} \leq C r M^{q/2} \quad\text{for all } x\in B_r(x_v),
\end{equation}
provided $B_{2r}(x_v)\subseteq \O.$
\end{lemma}

\noindent{\it Proof.} To derive these estimates we apply the inhomogeneous Harnack inequality (Theorem 4.17 in \cite{HanLin}) with the $L^2$-norm of the non-homogeneous term. For the solution $N-v\geq 0$ of $\D (N-v) = u^q$, satisfying $(N-v)(x_v)=0$, we get the bound
\[
\sup_{B_r(x_v)} (N-v) \leq c_0\Big(\inf_{B_r(x_u)} (N-v) + r \|u^q\|_{L^2(B_{2r}(x_v))}\Big) \leq c_0 r M^{q/2} \|u^{q/2}\|_{L^2(B_{2r}(x_v))}.
\]
Now \eqref{prop:int:eq2} of Proposition \ref{prop:int} allows us to estimate
\[
\sup_{B_r(x_v)} (N-v) \leq c_0 r M^{q/2} \|u^{q/2}\|_{L^2(B_{2r}(x_v))} \leq c_0 r M^{q/2} \left(\int_\O u^q \, dx \right)^{1/2} \leq Cr M^{q/2}.
\]
Analogously, for the solution $M-u\geq 0$ of $\D (M-u) = v$, satisfying $(M-u)(x_u)= 0$, we obtain
\[
\sup_{B_r(x_u)} (M-u) \leq c_0 r \|v\|_{L^2(B_{2r}(x_u))}.
\]
Now, the fact that
\[
\int_\O v^2 \, dx = \int_\O u^{q+1} \, dx \leq M \int_\O u^q \, dx \leq \tilde{C}
\]
allows us to conclude the proof, since
\[
\sup_{B_r(x_u)} (M-u) \leq c_0 r \|v\|_{L^2(B_{2r}(x_u))} \leq  c_0 r \tilde{C}^{1/2} \leq Cr.
\]

The second result we establish is that $M$ and $N$ are uniformly positive, and $\lim_{q\to\infty} M_q =1$.
\begin{prop}\label{prop:bb} Assume \eqref{eq:qk} and let $(u,v)$ be a classical solution of \eqref{LEbh} in a bounded $C^2$-domain $\O\subset \R^2$. Denote $M:=\max_{\overline{\O}} u$ and $N:=\max_{\overline{\O}} v$. There are constants $c_0,c>0$, such that
\begin{equation}\label{prop:bb:eq}
c\leq M \leq c_0 N.
\end{equation}
Furthermore, the following asymptotic bounds on $M=M_q$ hold:
\begin{equation}\label{prop:bb:asymp}
1 - \bar{c}/q \leq M_q \leq 1 + 4 \log q /q \quad \text{as }q\to \infty,
\end{equation}
for some $\bar{c}=\bar{c}(\O)>0$. In particular, $\lim_{q\to\infty} M_q =1.$
\end{prop}
\begin{proof}
Let $\l=\l(\O)$ be the first Dirichlet eigenvalue of $-\D$ in $\O$ and let $\phi>0$ be an associated eigenfunction. Multiplying both equations in \eqref{LEbh} by $\phi$ and integrating by parts, we get 
\[
\l^2 \int_\O u\phi \, dx = \l \int_\O v\phi  \, dx = \int_\O u^q \phi \, dx,
\]
so that $\int_\O (u^{q-1}- \l^2) u \phi \, dx  = 0$. Hence $M^{q-1}>\l^2$, meaning
\begin{equation}\label{prop:bb:Mlow}
M> \l^{2/(q-1)}  \geq \min\{1,\lambda^2\}  =c (\O)>0.
\end{equation}
The lower bound on $N\geq M/c_0\geq c/c_0$ is then precisely the $L^\infty$-estimate in Lemma \ref{lem:relativesizes} with $p=1$.

Let us now obtain the asymptotic behavior of $M$ as $q\to \infty$. The lower bound in \eqref{prop:bb:asymp} follows from \eqref{prop:bb:Mlow}:
\[
M\geq \exp \left(\frac{2 \log \l}{q-1}\right) \geq 1 + \frac{2 \log \l}{q-1} \geq 1- \frac{\bar{c}}{q} \quad \text{as } q\to \infty.
\]
In order to derive the upper asymptotic bound, we will first prove that
\begin{equation}\label{prop:bb:Mq}
M^q\leq C q^2 \quad \text{as } q\to \infty.
\end{equation}
Set $\rho:=M/q$. Since  $\rho\le C/q$ by \eqref{sec:BH:M} and $u$ achieves its maximum at a point $x_u\in \O$ which is at a  distance $d(x_u, \de \O)\geq \d_0(\O)>0$ away from $\de \O$ (according to Proposition \ref{prop:int}), the disk $B_{2\rho}(x_u)\subset \O$ for all sufficiently large $q$.
Because of the Harnack estimate \eqref{lem:inhHarnack2:1}, we see that for all large $q$,
\[
u^q(x) \geq (M- C\rho)^q = M^q (1- C/q)^q \geq  e^{-2C} M^q \quad \text{for all }  x \in B_\rho(x_u).
\]

Now the integral bound \eqref{prop:int:eq2} yields the desired
\begin{align*}
C\geq \int_\O u^q \, dx \geq \int_{B_{\rho}(x_u)} u^q \, dx \geq C' M^q |B_{\rho}(x_u)| =C'' M^q M^2/ q^2 \geq \tilde{C} M^q /q^2,
\end{align*}
where we used and the uniform bound  \eqref{prop:bb:Mlow}  from below $M\geq c$ to derive the last inequality. From \eqref{prop:bb:Mq} we then obtain
\[
M\leq \exp(\log(Cq^2)/q) \leq \exp (3 \log q/q)\leq 1 + 4 \log q/q \quad \text{as } q\to \infty.
\]
\end{proof}

We proceed with a key integral estimate.
\begin{prop}\label{prop:intbdsbelow} There exist constants $0<c<C$, depending only on $\O$, such that each of the quantities $\|u\|_{L^1(\O)}$, $\|u^q\|_{L^1(\O)}$, $\|v\|_{L^1(\O)}$, $\|v^2\|_{L^1(\O)}=\|u^{q+1}\|_{L^1(\O)}$ is bounded from below by $c$ and from above by $C$.
\end{prop}
\begin{remark}
Proposition \ref{prop:intbdsbelow} implies that $q\int_{\O} u^{q+1} \, dx \to \infty$ when $p=1$ in \eqref{LEsys}. This is in stark contrast to the regime $p\sim q$, in which the corresponding integral stays bounded (in star-shaped domains), according to Theorem \ref{coro_IntgBnd}.
\end{remark}

\begin{proof}
Let us first show that the quantities
\begin{equation}\label{prop:intbdsbelow:0}
\|u\|_{L^1(\O)} \sim \|u^q\|_{L^1(\O)} \sim \|v\|_{L^1(\O)} 
\end{equation}
are comparable with constants depending only on $\O$. Denote by $\phi$ the $L^1$-normalized first Dirichlet eigenfunction of $-\D$ in $\O$.  We have by Proposition \ref{prop:int}
\begin{align}\label{prop:intbdsbelow:1}
& \|u\|_{L^1} \gtrsim \|u\phi \|_{L^1} \gtrsim \|v\|_{L^1} \quad \text{and}\quad \|v\|_{L^1} \gtrsim \|v\phi\|_{L^1} \gtrsim \|u^q\|_{L^1}.
\end{align}
On the other hand, by the Divergence Theorem and the global Harnack inequality of Theorem~\ref{thm:GHI} applied to the superharmonic $u$ and $v$, we have
\begin{equation}\label{prop:intbdsbelow:2}
\begin{array}{c}
\|u^q\|_{L^1} = \int_{\O} (-\D v)\, dx = \int_{\de \O} (-v_\nu) \, ds \geq c \inf_{\de \O} (-v_{\nu}) \geq c \|v\|_{L^1}, \\ \\
\|v\|_{L^1} = \int_{\O} (-\D u)\, dx = \int_{\de \O} (-u_\nu) \, ds \geq c \inf_{\de \O} (-u_{\nu}) \geq c \|u\|_{L^1}.
\end{array}
\end{equation}
We see that estimates \eqref{prop:intbdsbelow:1} and \eqref{prop:intbdsbelow:2} entail  \eqref{prop:intbdsbelow:0}. The uniform upper bound on these three quantities follows from Proposition \ref{prop:int}.

To establish the uniform lower bound on all three, it suffices to show that $\|u\|_{L^1}\geq c$. Note that the Harnack  bound \eqref{lem:inhHarnack2:1} of Lemma \ref{lem:inhHarnack2} says that
\begin{equation}\label{prop:intbdsbelow:harnack}
u(x) \geq M - c' r \quad \text{for } x\in B_{r}(x_u), \quad \text{provided } B_{2r}(x_u)\subset \O,
\end{equation}
where $x_u$ is a point of maximum of $u$. Taking into consideration the uniform lower bound \eqref{prop:bb:eq} on $M$ and the fact that $d(x_u,\de \O) \geq \d_0(\O)>0$ (Proposition \ref{prop:int}), we see that \eqref{prop:intbdsbelow:harnack} yields a neighborhood $B_{\bar{c}}(x_u)\subset \O$ where
\begin{equation*}
u(x) \geq M/2 \geq c_0 \quad \text{for all } x\in B_{\bar{c}}(x_u).
\end{equation*}
Thus,
$$\|u\|_{L^1} \geq \int_{B_c(x_u)} u \, dx \geq c_0 |B_{\bar{c}}(x_u)|\geq c.$$
Finally, to show that the remaining quantity $\|v^2\|_{L^1}=\|u^{q+1}\|_{L^1}$ is of unit size, we note that by Proposition \ref{prop:int} and \eqref{sec:BH:M}
\begin{align*}
\|u^{q+1}\|_{L^1} & \leq M \|u^q\|_{L^1} \leq C, \\
\|v^2\|_{L^1} & \geq |\O|^{-1} \|v\|_{L^1}^2 \geq c,
\end{align*}
where we used the Cauchy-Schwarz inequality to derive the second estimate above.
\end{proof}

We are now in a position to establish the logarithmic growth of $\max_{\overline{\O}} v_q$ as $q\to\infty$.
\begin{theorem}\label{thm:v_blowup} Let $(u_q,v_q)$ be a classical solution of \eqref{LEbh} in $\O$. Then
\begin{equation}\label{thm:v_blowup:asymp1}
 \max_{\overline{B_1}} v_q \leq \bar{C} \log q \quad \text{as } q\to \infty.
\end{equation}
while if $\O=B_1(0)\subset \R^2$ then for some absolute constant $\bar{c}>0$
\begin{equation}\label{thm:v_blowup:asymp2}
\max_{\overline{B_1}} v_q\geq \bar{c} \log q \quad \text{as } q\to \infty.
\end{equation}
\end{theorem}
\begin{proof}
 We will first show that the upper asymptotic bound in \eqref{thm:v_blowup:asymp1} holds true for any bounded $C^2$-domain $\O\subset \R^2$. For this purpose, we shall make use of the $L^1$ estimate of Brezis-Merle \cite{brezis1991uniform} for the Poisson equation in two dimensions, which we state now.

\begin{lemma}[\cite{brezis1991uniform}] \label{lem:BM}
Assume $\O\subset \R^2$ is a bounded domain and let $u$ be a solution of
\[
-\D u = f  \text{ in } \O, \quad
u = 0  \text{ on } \de \O,
\]
with $f\in L^1(\O)$. Then for every $\d\in (0,4\pi)$ we have
\[
\int_\O \exp \left(\frac{(4\pi - \d) |u|}{\|f\|_{L^1(\O)}}\right) \, dx \leq \frac{4\pi^2}{\d} (\text{diam } \O)^2.
\]
\end{lemma}
Applying Lemma \ref{lem:BM} to $-\D v = u^q$ in $\O\subset \R^2$, with $\d = 2\pi$, we get
\begin{equation}\label{thm:v_blowup:BM1}
C= 2\pi (\text{diam }\O)^2 \geq \int_\O \exp (2 \pi v/ \|u^q\|_{L^1(\O)}) \, dx \geq \int_\O e^{c v} \, dx,
\end{equation}
since $\|u^q\|_{L^1(\O)}\leq C$. However, by the $v$-estimate in Lemma \ref{lem:inhHarnack} we have near the point of maximum $x_v$ of $v$, by  \eqref{main:vest} with $p=1$,
\begin{equation}\label{thm:v_blowup:BM2}
v(x) \geq N - N/4 = 3N/4 \quad \text{for all } x \in B_{c_0 \rho}(x_v)\subset \O, \text{ with } \rho := (N/M^q)^{1/2}.
\end{equation}
Combining \eqref{thm:v_blowup:BM1}, \eqref{thm:v_blowup:BM2} and the fact that $N\geq c$ by Proposition \ref{prop:bb}, we obtain
\begin{equation}\label{thm:v_blowup:Mqexpo}
C\geq \int_\O e^{c v} \, dx \geq \int_{B_{c_0\rho}(x_v)} e^{c v} \, dx \geq e^{\tilde{c} N} |B_{c_0\rho}|  = c' e^{\tilde{c} N} N/M^q \geq c'' e^{\tilde{c}N}/M^q.
\end{equation}
But we know from \eqref{prop:bb:Mq} that $M^q \leq C q^2$ as $q\to \infty$. Hence, $e^{\tilde{c} N} \leq \tilde{C}q^2$ and we can conclude \begin{equation}\label{boundN}N\leq \bar{c} \log q \quad \text{as } q\to \infty.\end{equation}\medskip

We now turn our attention to proving the logarithmic growth of $N_q$ in \eqref{thm:v_blowup:asymp2}.
From here on we assume that $\O=B_1(0)$. Because of the radial symmetry of the domain, the moving planes method yields that $u(x)$ and $v(x)$ are radially symmetric and monotonically decreasing along the radius, so that
\[
M:=\max_{\overline{B_1(0)}} u = u(0) \quad \text{and} \quad N:= \max_{\overline{B_1(0)}} v = v(0).
\]
In particular, \emph{the points of maximum of both solution components $u$ and $v$ coincide}: a significant feature which will be exploited in the upcoming arguments. 

\bigskip

\noindent\emph{Step 1.} We claim that there exists a large  constant $L>0$ and a constant $c_0>0$ such that
\begin{equation}\label{thm:v_blowup:massconc}
\int_{\{x\in \O: u(x)\geq 1- L/q\}} u^q \, dx \geq c_0\quad \text{for all large } q.
\end{equation}
Indeed, when $u(x)< 1-L/q$, $L\geq1$, we have $u^q(x) \leq e^{-L}$, so that
$
\int_{\{u< 1- L/q\}} u^q \, dx \leq e^{-L} |\O|.
$
Since by Proposition \ref{prop:intbdsbelow} we know that $\int_\O u^q \, dx \geq c$, we see that
\[
\int_{\{u\geq 1- L/q\}} u^q \, dx \geq c - e^{-L} |\O| \geq c/2,
\]
if we choose $L = \max(1, \log (2|\O|/c))$.

\bigskip

\noindent \emph{Step 2}. Denote  
$$
I(r):=\int_{B_{2r}} u^q\, dx, \quad r\in(0,1/2].
$$
Fix any $R\in (0,1/2]$. We will show that in $B_R$, $u$ \emph{decreases quadratically} away from its point of maximum, according to
\begin{equation}\label{thm:v_blowup:quadgrowth}
M- u (x) \geq  N |x|^2/4 - C M^{q/2} \sqrt{I(R)} |x|^3 \quad \text{for all } x\in \overline{B_R}.
\end{equation}
Indeed, since the origin is the point of maximum of $v$, we see by the $v$-estimate \eqref{lem:inhHarnack2:2} of Lemma \ref{lem:inhHarnack2}, applied at every $x$ with $r=|x|\in (0,R]$, that
\[
v(x)  \geq N  - c_0 r M^{q/2}\|u^{q/2}\|_{L^2(B_{2r})}= N- c_0 M^{q/2} \sqrt{I(R)}|x|  \quad \text{for all } x\in \overline{B_R}.
\]
Thus, if
$$
\tilde{u}_0(x) := N |x|^2 /4 - c_0 M^{q/2} \sqrt{I(R)} |x|^3 /9,$$
$$
w:=M- u - \tilde{u}_0,$$
then the calculation
\[
\D w (x) = v(x) - (N- c_0 M^{q/2} \sqrt{I(R)}|x|) \geq 0  \quad \text{for all } x\in B_R,
\]
shows that $w$ is subharmonic in $B_R$. Hence, radial symmetry and the mean value property yield for each $x$ with $r=|x|\in (0,R]$:
\[
M-u(x) - \tilde{u}_0(x) = w(x) = \frac{1}{2\pi} \int_{\de B_r} w \, ds \geq w(0) = 0,
\]
whence $M-u\geq \tilde{u}_0$ in $B_R$, which is precisely \eqref{thm:v_blowup:quadgrowth}.

\bigskip

\noindent \emph{Step 3}. In this step we will show that
\begin{equation}\label{thm:v_blowup:Mqinf}
\frac{M^q}{N^2} \geq c q^{1/2}.
\end{equation}

 First, we observe that if  $x\in S:=\{x\::\:u(x) \geq 1- L/q\}$, where $L$ is the constant from Step 1, then by \eqref{prop:bb:asymp}
\begin{equation}\label{thm:v_blowup:1}
M-u(x)\le 4\frac{\log q}{q}+\frac{L}{q}.
\end{equation}
On the other hand, estimate \eqref{thm:v_blowup:quadgrowth} for $R=1/2$ implies that for all $x\in \overline{B_{1/2}}$,
$$
M-u(x)  \geq  |x|^2 (N/4 - CM^{q/2} \sqrt{I(1/2)} |x|) \geq  |x|^2 (N/4 - CM^{q/2}|x|),
$$
since $ I(1/2) = \int_{B_1} u^q \leq C$.
Hence for all $x\in S$ such that $|x|\leq N/(8CM^{q/2})$ we have by~\eqref{prop:bb:eq}
\begin{equation}\label{thm:v_blowup:2}
c_0 |x|^2 \le (N/8) |x|^2\le 4\log q/q+L/q
\end{equation}
 Assume for contradiction that $N/M^{q/2}\geq A q^{-1/4}$, where $A>0$ is picked so that the set
 $$
 \left(\frac{4\log q}{c_0 q} + \frac{L}{c_0 q}\right)^{1/2} < |x| < \frac{A}{2Cq^{1/4}} \left( \leq \frac{N}{2CM^{q/2}}\right)
 $$
 is not empty for all $q\geq 2$. However, by \eqref{thm:v_blowup:2}  no point of this set can be in $S$. Since $u$ is radially decreasing, this means that $S\subset B_\rho$, where $\rho =((4/c_0)\log q/q + L/(c_0q))^{1/2}$. However, that and the logarithmic bound \eqref{boundN} for $N$ entail
\begin{align*}
\int_{\{u\geq 1- L/q\}} u^q \, dx  & \leq \int_{B_{\rho}} u^q \, dx \leq M^q |B_{\rho}| = C_1 \frac{M^q}{N^2} N^2\rho^2 \\
& \leq C A^{-2} q^{1/2} (\log q)^2 \left(\frac{4\log q}{c_0 q} + \frac{L}{c_0 q}\right)  \to 0 \quad \text{as }q\to\infty,
\end{align*}
which contradicts the mass concentration estimate~\eqref{thm:v_blowup:massconc}.

\bigskip

\noindent \emph{Step 4.} In this final step we will establish the precise, logarithmic growth of $N_q$.
According to \eqref{thm:v_blowup:Mqinf} from the previous step, we know that
\begin{equation}\label{thm:v_blowup:Rdef}
R:= A N/M^{q/2}  \leq C q^{-1/2}  \quad \text{as } q\to \infty,
\end{equation}
where $A>0$ is a large absolute constant to be picked later in the argument. Now, if one has $I(R)> 1/A^4$, then Green's representation formula (see \eqref{main:Green_u}) yields
\begin{align*}
 N + C'= v(0)+C' &\geq  \int_{B_1} \frac{1}{2\pi} \log \frac{1}{|y|} u^q(y) \, dy \geq \frac{1}{2\pi} \log\frac{1}{2R} \int_{B_{2R}} u^q \, dy  \\
& = \frac{1}{2\pi} \log\left(\frac{1}{2R}\right) I(R) \geq \frac{1}{2\pi A^4} \log\left(\frac{1}{2R}\right)  \geq c\log (q)    \quad \text{as } q\to \infty,
\end{align*}
due to \eqref{thm:v_blowup:Rdef},  so $ N\geq c\log (q) - C' \geq (c/2) \log (q)$ as $q\to\infty$, and we are done.

Therefore, we may assume that $I(R)\leq 1/A^4$.
Then the quadratic decrease estimate \eqref{thm:v_blowup:quadgrowth} of Step 2 yields
\begin{align}\label{thm:v_blowup:quaddec}
M- u (x) & \geq N |x|^2/4 - C M^{q/2} A^{-2} |x|^3  \geq |x|^2 (N/4 - C M^{q/2} A^{-2} R) \notag \\
& \geq N|x|^2 (1/4 - C/A) \geq N|x|^2/8 \quad \text{for all }x\in \overline{B_R},
\end{align}
provided $A\geq 8C$. In that case, because of the radial monotonicity of $u$, we  have that
\begin{equation}\label{thm:v_blowup:Raway}
M-u(x) \geq  N|R|^2/8  = c A^2 N^3/M^q \quad \text{for all } x\in B_1\setminus \overline{B_R}.
\end{equation}
In order to estimate the right-hand side, we will use the bound \eqref{main:ineq_v} from the proof of Theorem \ref{main}. It states that for some absolute constants $c_1, c_2, c_3$ and all large $q$
\begin{equation}\label{thm:v_blowup:ineq_N}
\frac{(N+c_1)N}{M^{q+1}} \geq \frac{c_2}{q} \log \left(c_3 \frac{q N}{M}\right).
\end{equation}
By Proposition \ref{prop:bb} we know that $N\geq \tilde{c}_1$ and $1/2\leq M \leq 2$ for all large $q$. Hence \eqref{thm:v_blowup:ineq_N} entails that, for all large $q$, the right-hand side of \eqref{thm:v_blowup:Raway} is bounded from below by
\begin{align*}
c A^2 \frac{N^3}{M^q}  & \geq c \tilde {c}_1 A^2 \frac{N^2}{M^q}  \geq \tilde{c} A^2 \frac{(N+c_1)N}{M^{q}} \geq \tilde{c} A^2 \frac{c_2 M}{q} \log \left(c_3 \frac{qN}{M}\right) \\
& \geq \frac{c_0  A^2}{q} \log (c_0' q) \geq 8 \frac{\log q}{q}, \quad \text{provided} \quad c_0 A^2 \geq 9.
\end{align*}
Hence, choosing
\[
A = \max((9/c_0)^{1/2}, 8C)
\]
we can conclude from \eqref{thm:v_blowup:Raway} that
\begin{equation}\label{thm:v_blowup:estu}
M- u(x) \geq 8 \log q /q \quad \text{whenever } x\in B_1\setminus \overline{B_R}, \quad \text{as }q\to \infty.
\end{equation}
But for all large $q$, by using \eqref{prop:bb:asymp} in Proposition \ref{prop:bb},
\[
S:=\{u\geq 1- \frac{L}{q}\} = \{M-u \leq M-1+\frac{L}{q}\} \subseteq \{M-u\leq \frac{4 \log q + L}{q} \} \subseteq \{M-u< \frac{8 \log q}{q} \},
\]
so that \eqref{thm:v_blowup:estu} implies that $S \subseteq  \overline{B_R}$. By this and Step 1 we infer that
\[
\int_{B_{R}} u^q \, dx \geq \int_{S} u^q \, dx \geq c_0,
\]
so that another application of Green's formula yields once again
\[
 N +C' = v(0) +C'\geq \frac{1}{2\pi} \log \left(\frac{1}{R}\right) \int_{B_R} u^q \, dx \geq \frac{c_0}{2\pi} \log \left(\frac{1}{R}\right)  \geq c\log (q)   \quad \text{as } q\to\infty.
\]

Theorem \ref{biheq} is proved.
\end{proof}

\bigskip

In the end we display a Maple-generated plot of the maxima of $u$ and $v$, and their asymptotics, for the biharmonic Lane-Emden equation $-\Delta u =v$, $-\Delta v= u^q$, in the unit disk, with Navier boundary conditions.

\begin{figure}[h]

\begin{subfigure}{0.45\textwidth}
\includegraphics[width=0.9\linewidth, height=6cm]{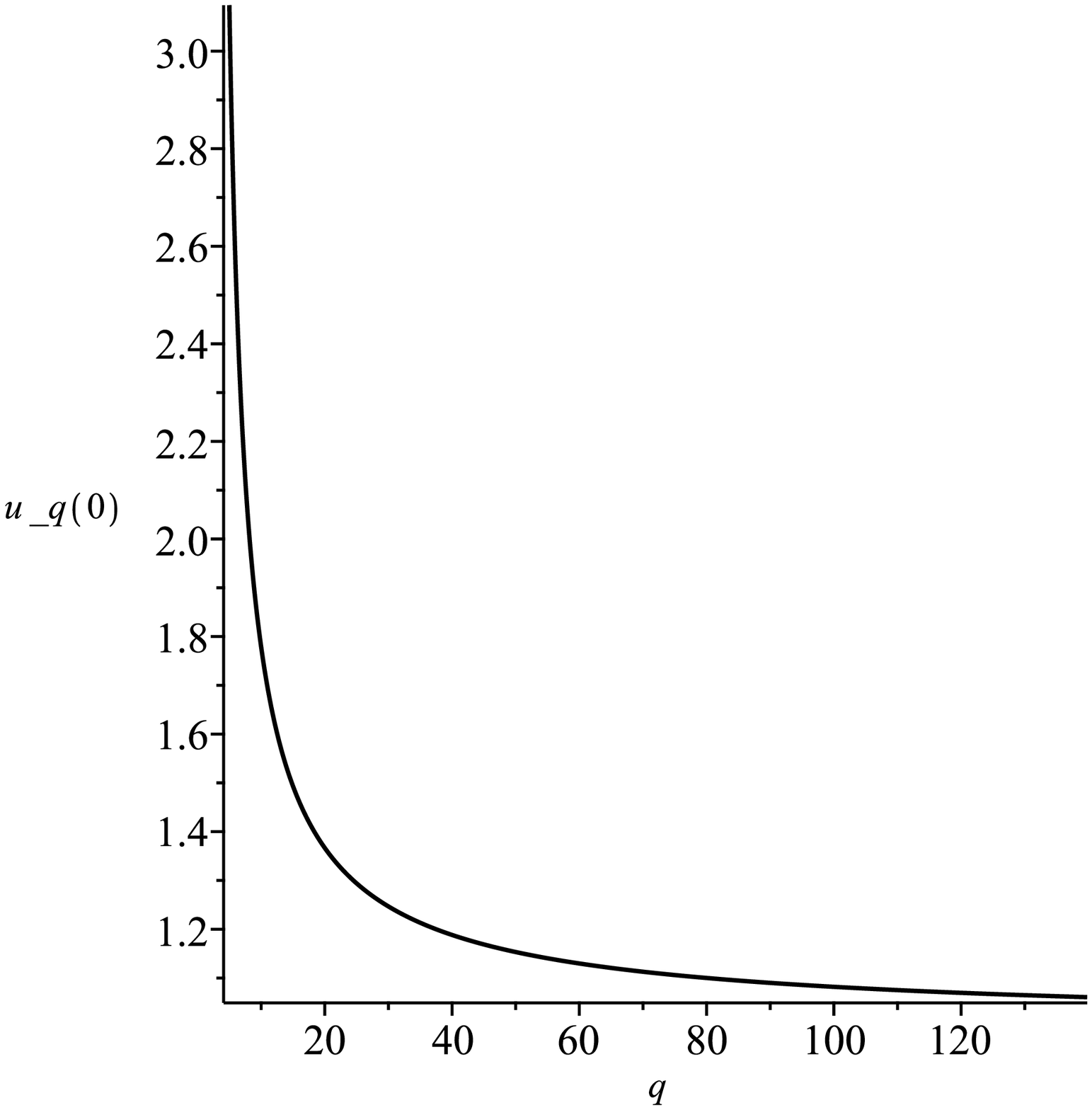}
\caption{$\max(u_q)$ vs. $q$}
\label{fig:subim1}
\end{subfigure}
\begin{subfigure}{0.45\textwidth}
\includegraphics[width=0.9\linewidth, height=6cm]{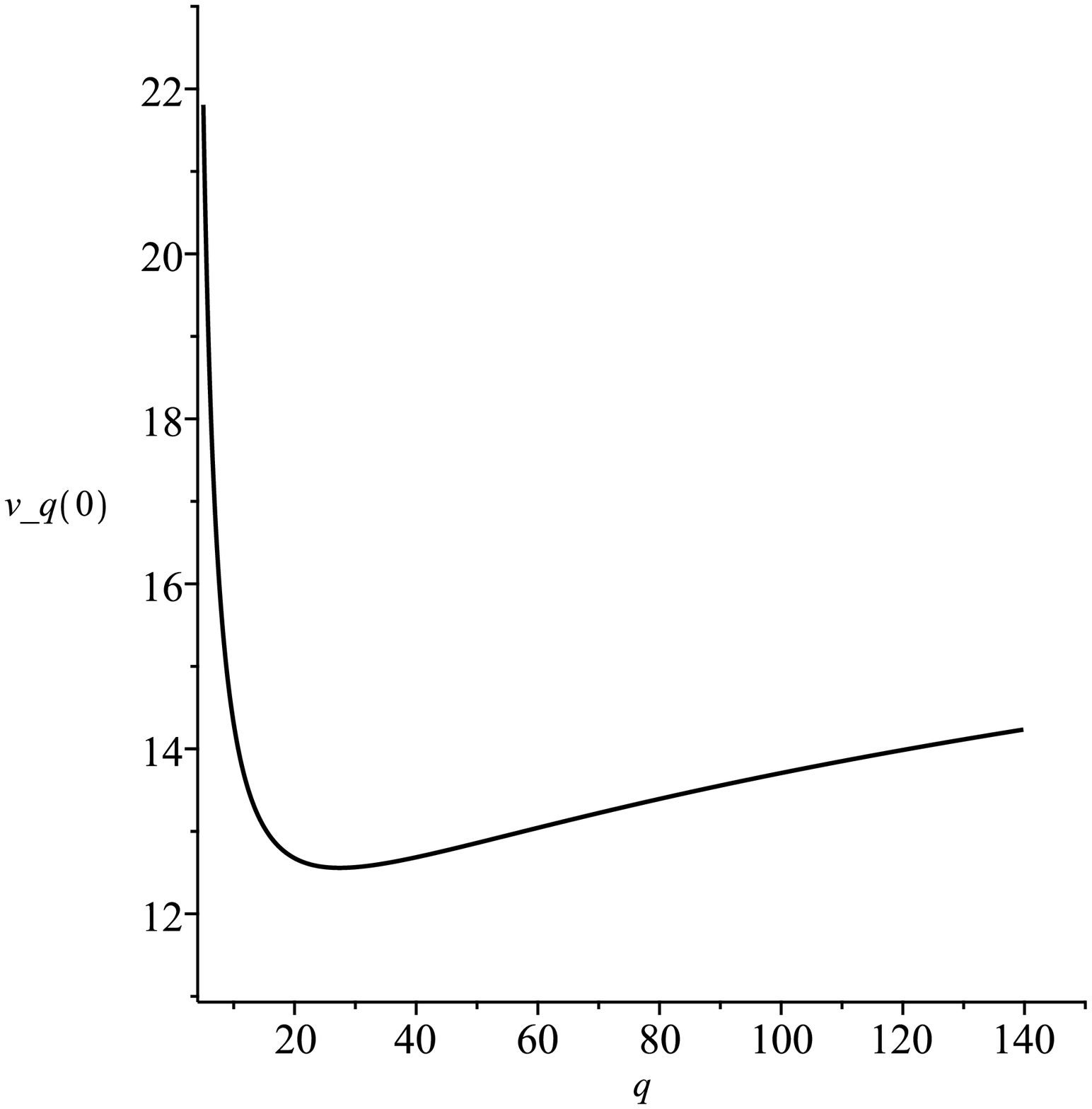}
\caption{$\max(v_q)$ vs. $q$}
\label{fig:subim2}
\end{subfigure}

\caption{Maxima of the solution components}
\label{fig:image1}
\end{figure}

\begin{figure}[h]

\begin{subfigure}{0.45\textwidth}
\includegraphics[width=0.9\linewidth, height=6cm]{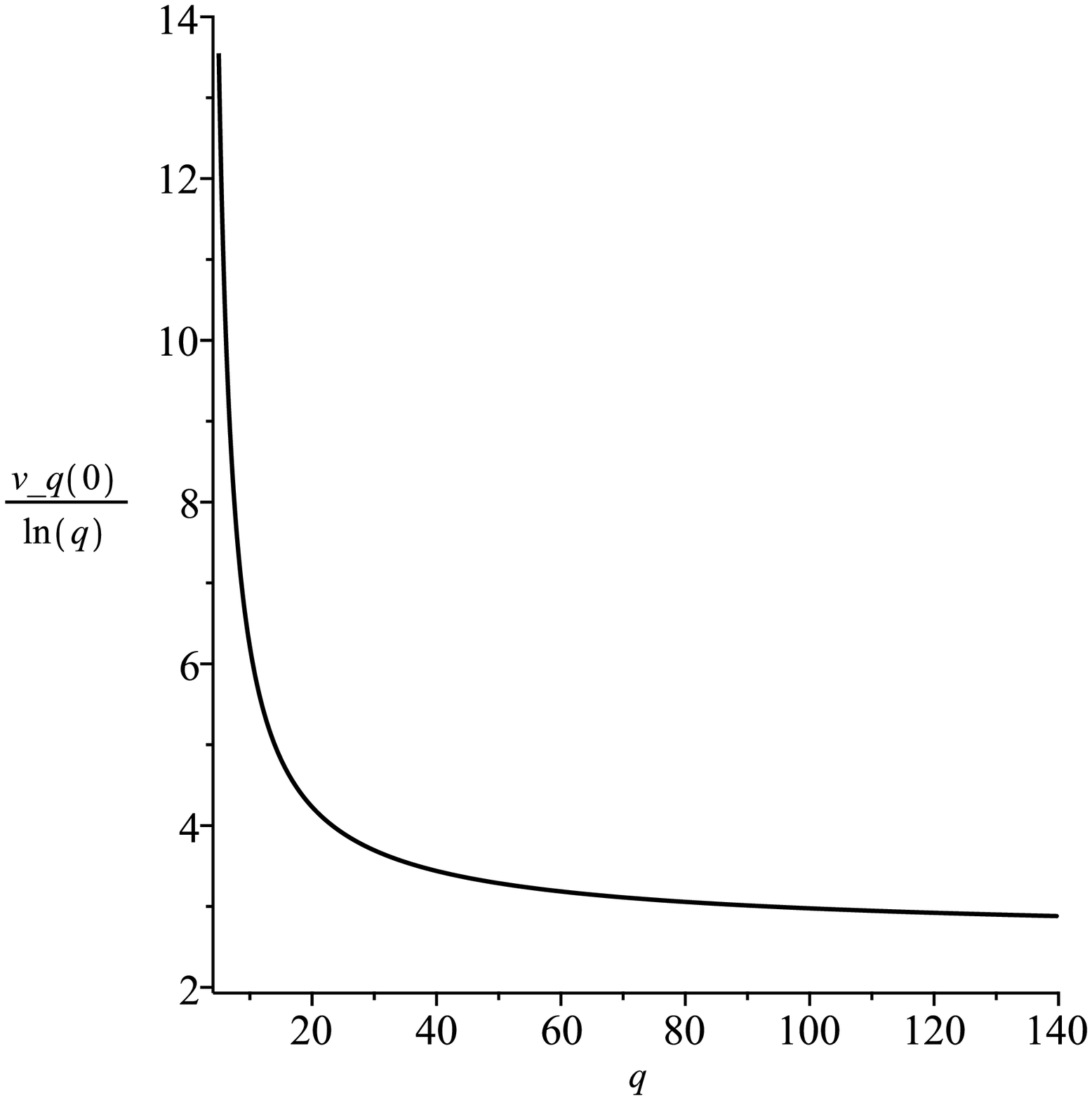}
\caption{$\frac{\max(v_q)}{\log(q)}$ vs. $q$}
\label{fig:subim3}
\end{subfigure}
\begin{subfigure}{0.45\textwidth}
\includegraphics[width=0.9\linewidth, height=6cm]{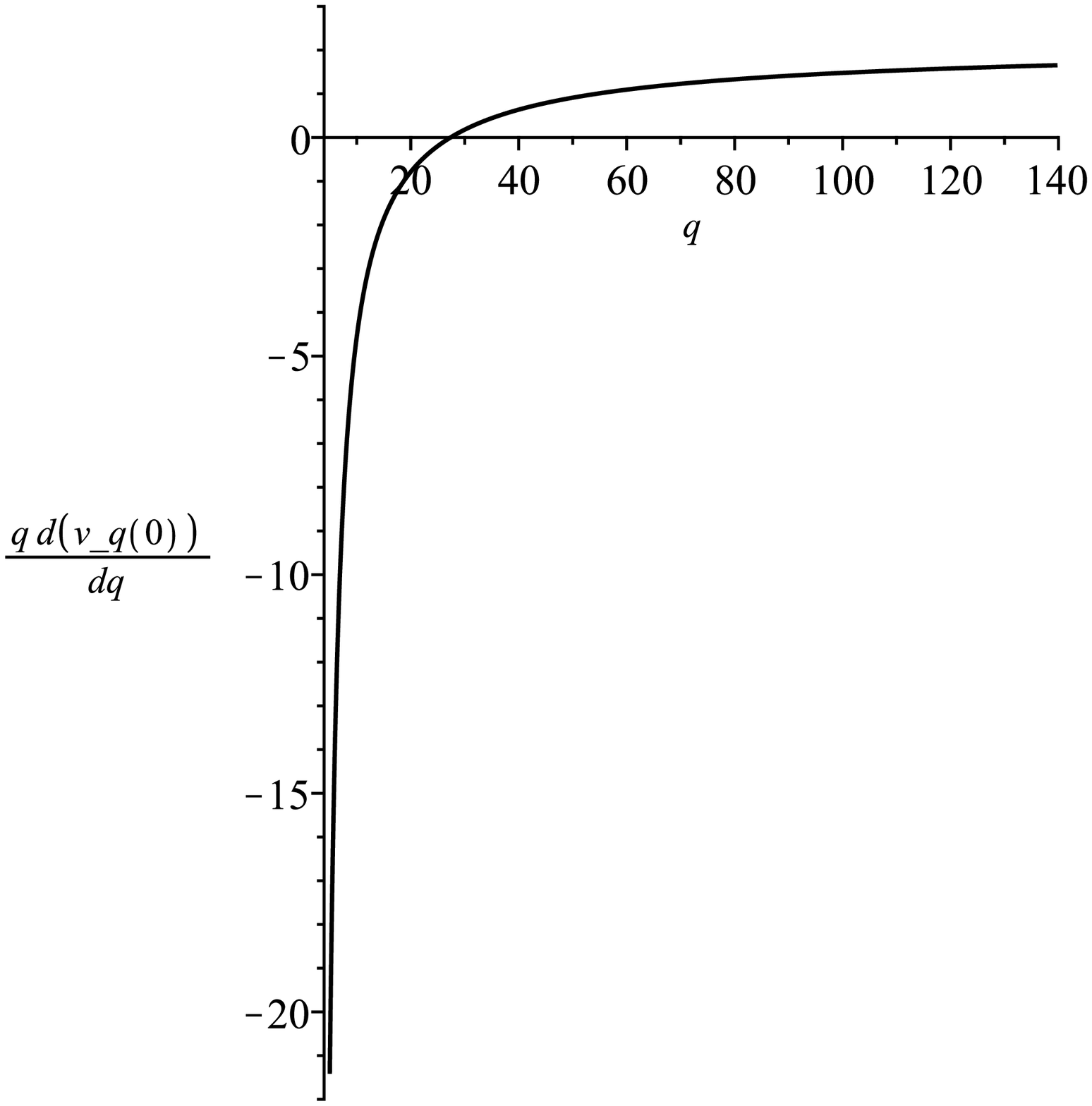}
\caption{$q\displaystyle\frac{\partial(\max(v_q))}{\partial q}$ vs. $q$}
\label{fig:subim4}
\end{subfigure}

\caption{Logarithmic asymptotics of the maximum of the $v$-component}
\label{fig:image2}
\end{figure}

\bibliography{LESys2D_Biblio}
\end{document}